\documentclass[10pt]{amsart}
\usepackage{setspace,lmodern,xypic,somedefs,tracefnt,latexsym,rawfonts,graphicx,amsfonts,amssymb,latexsym,enumerate,amscd,hyphenat,amsthm}
\usepackage[all]{xy}
\def\cal{\mathcal}
\def\Bbb{\mathbb}

\def\r{\rangle}
\def\l{\langle}

\def\bar{\overline}

\newtheorem{thm}{Theorem}[section]
\newtheorem{prop}[thm]{Proposition}
\newtheorem{exm}[thm]{Example}
\newtheorem{lemma}[thm]{Lemma}
\newtheorem{cor}[thm]{Corollary}
\newtheorem{defn}[thm]{Definition}
\newtheorem{rem}[thm]{Remark}

\numberwithin{equation}{section}
\begin{document}
\date{\today}
\title[Fundamental groups of orbit configuration spaces]
{A four-term exact sequence of fundamental\\ groups of orbit configuration spaces}
\author[S.K. Roushon]{S.K. Roushon}
\address{School of Mathematics\\
Tata Institute\\
Homi Bhabha Road\\
Mumbai 400005, India}
\email{roushon@math.tifr.res.in} 
\urladdr{http://mathweb.tifr.res.in/\~\!\!\! roushon/}
\begin{abstract} We deduce that the
  fundamental groups of the orbit configuration spaces of an effective and properly
  discontinuous action of a discrete group on a connected aspherical
  $2$-manifold, with isolated fixed points, fit into a four-term
  exact sequence. This comes as a consequence of the four-term exact
  sequence of orbifold pure braid groups (\cite{Rou20}, \cite{JF23} and \cite{Rou24-1}).
  The proof relates these two exact sequences and also draws a new
  consequence (Corollary \ref{free}) on the later one.\end{abstract}
 

\maketitle

\section{Introduction}
We recall the definition of an orbit configuration space.

\begin{defn}\label{ecs}{\rm (\cite{Xic97}) Let $G$ be a discrete group
    acting on a topological space $X$, and $n\geq 1$
    be an integer.
    The {\it orbit configuration space}
    ${\cal {O}}_n(X,G)$ of the pair $(X,G)$ is the space of all $n$-tuple
    of points of $X$ with distinct orbits. That is, 
    $${\cal O}_n(X,G)=\{(x_1,x_2,\ldots, x_n)\in X^n\ |\ Gx_i\cap Gx_j=\emptyset,\ \text{for}\ i\neq j\}.$$
    If $G$ is the trivial group,
    then ${\cal O}_n(X,G)$ is abbreviated as ${\cal O}_n(X)$, and it is the
 usual configuration space (\cite{FN62}) of ordered $n$-tuple of distinct points of $X$.
 $*_n\in {\cal O}_n(X,G)$ will denote a point 
 whose coordinates have trivial isotropy
groups. $\bar *_n$ will denote its image in ${\cal O}_n(X/G)$.}\end{defn}

In this note, we always consider an effective and properly
discontinuous (\cite{Thu97}, Definition 3.5.1) action of $G$ on a
smooth connected manifold $X$ without boundary.

 First, we recall the celebrated Fibration Theorem of Fadell and
Neuwirth for configuration spaces. Consider the projection
  $p:{\cal O}_n(X, G)\to {\cal O}_{n-1}(X, G)$ to
  the first $n-1$ coordinates and let $p(*_n)=*_{n-1}$.

\begin{thm}{\rm (\cite{FN62}, Fibration Theorem)} \label{FN}
  $p:{\cal O}_n(X)\to {\cal O}_{n-1}(X)$ is a locally trivial fibration with
  fiber $X-\{(n-1)\ \text{points}\}$.\end{thm}

In \cite{Xic97}, the following fibration theorem was established
for orbit configuration spaces.

\begin{thm}{\rm (\cite{Xic97})}\label{Xic} If the action of $G$ on $X$
  is properly discontinuous and free then $p:{\cal O}_n(X, G)\to {\cal
    O}_{n-1}(X, G)$ is a locally trivial fibration with
  fiber $X-\{(n-1)\ {\text orbits}\}$.\end{thm}

The class of orbit configuration spaces is a very important class, and
appears in many branches of Mathematics and Physics
(\cite{Arn69} - \cite{Bot96}, \cite{BP02},
\cite{Coh76}, \cite{Coh95}, \cite{Nam49} and \cite{Tot96} -
\cite{Vas92}). Much work has been done in this 
area during the last couple of decades, in the particular case when
the action is properly discontinuous and free. See \cite{Xic14} and \cite{LS05} for some 
works and relevant literature on orbit configuration spaces.
Also see \cite{CLW21} and \cite{BiGa18} for some non-free action cases.
However, most of these studies are related to cohomological computations or to
know the homotopy type of the orbit configuration spaces. 
The fundamental groups of orbit configuration
spaces are interesting only in the cases of $2$-manifolds.
Since, ${\cal O}_n(X,G)$ is the complement of codimension $\geq 3$ submanifolds
of $X^n$ if dimension of $X$ is $\geq 3$, and
hence $X^n$ and ${\cal O}_n(X, G)$ will have isomorphic fundamental groups.

Applying the long exact homotopy sequence induced by a fibration and
by an induction argument, one can
deduce the following two conclusions from Theorem \ref{Xic}.

\begin{cor}\label{corollary} Consider a connected $2$-manifold
  $S\neq {\Bbb {RP}}^2, {\Bbb S}^2$, and assume that 
the action of $G$ on $S$ is properly discontinuous and free. Then,  

$\bullet$ ${\cal O}_n(S,G)$ is aspherical, that is $\pi_i({\cal O}_n(S,G))=\l 1\r$, for
all $i\geq 2$ and 

$\bullet$ there is a short exact sequence:

\centerline{
  \xymatrix{
    1\ar[r]&\pi_1(F, *_n)\ar[r]^{\!\!\!\!\!\!\!\!\! i_*}&
    \pi_1({\cal O}_n(S, G), *_n)\ar[r]^{\!\!\!\!\!\! p_*}&
    \pi_1({\cal O}_{n-1}(S, G), *_{n-1})\ar[r]&1.}}
Here, $F=p^{-1}(*_{n-1})=S-\{(n-1)\ {\text orbits}\}$ and $i:F\to {\cal
  O}_n(S, G)$ is the inclusion map.\end{cor}

 We want to study the existence of the above type of short exact sequence
 when the action is not free. We will show that in this case
 $i_*$ is never injective. However, we have a hope that, even when the action is
 not free, asphericity of $S$ is preserved in ${\cal O}_n(S,G)$
 (\cite{Rou24-2}, Asphericity conjecture). See Remark
 \ref{asphericityC} for some examples when this is true.
 
Since the discrete group $G$ is
acting on $X$ effectively and properly discontinuously, $X/G$ has an orbifold structure and
$X\to X/G$ is an orbifold covering map (\cite{Thu91}, Proposition 5.2.6). An effective
action of a finite group on $X$ has this property. See Section
Appendix for a short 
introduction on orbifolds. Before we state our main result, recall that the underlying
topological space of a $2$-dimensional orbifold has a
manifold structure (\cite{Sco83}, p. 422, last para.). Hence, the {\it genus} of a
connected $2$-dimensional orbifold is defined as the genus of its
underlying space.

 We prove the following four-term exact sequence induced by $p$.
    
\begin{thm}\label{mt} Let $S$ be a connected $2$-manifold without boundary.
  Let a discrete group $G$ is acting on $S$ effectively, properly discontinuously
  and with isolated fixed points. In addition, 
  assume that $S\neq {\Bbb S}^2$ or ${\Bbb {RP}}^2$ and
  the underlying space of $S/G$ is ${\Bbb S}^2$ with at least one puncture, or of
    genus $\geq 1$ and orientable.
Then, $p$ induces the following four-term exact sequence.

\centerline{
  \xymatrix@C-=0.45cm{1\ar[r]&L(S, n-1)\ar[r]&\pi_1(F, *_n)\ar[r]^{\!\!\!\!\!\!\!\!\!\!\!\!\!\! i_*}&
    \pi_1({\cal O}_n(S, G), *_n)\ar[r]^{\!\!\!\!\!\! p_*}&\pi_1({\cal O}_{n-1}(S, G),*_{n-1})\ar[r]&1.}}
\noindent

Furthermore, $L(S, n-1)=\l 1\r$ if and only if $G$ also acts freely on $S$.
\end{thm}

Since we expect that (\cite{Rou24-2}, Asphericity conjecture) 
${\cal O}_n(S, G)$ is aspherical, therefore $\pi_1({\cal O}_n(S, G), *_n)$
should be torsion free.

Hence, we can state the following weaker conjecture.

\medskip
{\noindent}
{\bf Torsion-Free conjecture.} {\it Let $S$, $G$ and $F$ be as in the statement
of Theorem \ref{mt}. Then $\pi_1({\cal O}_n(S, G), *_n)$ is torsion
free and consequently, $\pi_1(F, *_n)/L(S,
n-1)\simeq {\text {ker}}(p_*)$ is torsion free.}
\medskip

\begin{rem}\label{asphericityC}{\rm Recall that the orbifold
    fundamental group  $\pi_1^{orb}({\cal O}_n(S/G))$
    of the configuration space ${\cal O}_n(S/G)$
    of the orbifold $S/G$ is called the {\it orbifold pure braid group} (\cite{Rou24-1}, Definition 2.1) of
    $S/G$, and $\pi_1({\cal O}_n(S,G))$ is isomorphic to a subgroup of 
    $\pi_1^{orb}({\cal O}_n(S/G))$ (see Example \ref{cov-map-example}). Although, $\pi_1^{orb}({\cal
      O}_n(S/G))$ may have torsion element, we expect in the above
    conjecture that $\pi_1({\cal O}_n(S,G))$ is torsion free.
    In [\cite{Rou24-2}, Theorem 1.2] we verified the Asphericity 
conjecture when the underlying space of $S/G$ is ${\Bbb C}$ with one cone point of order $\geq
2$, or with two cone points of order $2$ each, or the underlying space
of $S/G$ is ${\Bbb C}-\{0\}$ with one cone point of order
$2$. These three cases are related to the well-known $K(\pi, 1)$-problem for Artin
groups. See \cite{Rou24-2} for more details.}\end{rem}

\section{Orbifold pure braid groups and the snake lemma}
We recall that in [\cite{Rou20}, Lemma 2.9] we had constructed
homomorphisms between configuration Lie groupoids
(\cite{Rou20}, Definition 2.8) of a large class of Lie
groupoids. The orbit configuration space ${\cal O}_n(X,G)$ was the object space of the
configuration Lie groupoid, defined in [\cite{Rou20}, Definition 2.6],
of the translation Lie groupoid, corresponding to the
action of $G$ on $X$ (\cite{Rou20}, Example 2.2).
And $p$ was the object level map of this homomorphism. 
This homomorphism induced map on their
classifying spaces are not quasifibrations (in the sense of \cite{DT58}), if the action
has a fixed point (see Proposition \ref{finallemma}). However,
this homomorphism induces a four-term exact sequence of 
fundamental groups when $X$ is a $2$-manifold.

Let ${\cal C}$ be the class of
  all $2$-dimensional orbifolds $O$ with only cone points and satisfy the following
  conditions.

  $\bullet$ The underlying space of $O$ is the $2$-sphere with at least one puncture.

  $\bullet$ If $O$ has genus $\geq 1$ then it is orientable. 

  We now recall the main ingredient of the proof of Theorem \ref{mt}.
  For $O\in {\cal C}$, the
projection to the first $n-1$ coordinates $\bar p:{\cal O}_n(O)\to {\cal
  O}_{n-1}(O)$, induces a four-term exact sequence of their orbifold
fundamental groups.

    \begin{thm}\label{skr-jf}
      (\cite{Rou20},\cite{JF23},\cite{Rou24-1})
      Let $O\in {\cal C}$ and $\bar F={\bar p}^{-1}(\bar *_{n-1})$. Then there is the following
      exact sequence.
      
  \centerline{
  \xymatrix@C-=0.4cm{
    1\ar[r]&K(O, n-1)\ar[r]&\pi_1^{orb}(\bar F, \bar *_n)\ar[r]^{\!\!\!\!\!\!\!\!\!\! \bar i_*}&
    \pi_1^{orb}({\cal O}_n(O), \bar *_n)\ar[r]^{\!\!\!\!\!\! {\bar p_*}}&
    \pi_1^{orb}({\cal O}_{n-1}(O), \bar *_{n-1})\ar[r]&1.}}
Furthermore, $K(O,n-1)=\l 1\r$ if and only if $O$ has no singular point.
\end{thm}

From the proof of Theorem \ref{mt} we will deduce the following.

\begin{cor}\label{identify} Let $S$ and $G$ be as in the statement of Theorem \ref{mt},
  and $O=S/G$, then $L(S, n-1)\simeq K(O, n-1)$.\end{cor}

The following consequence of Theorem \ref{mt} and Corollary
\ref{identify} is new, that is, it does not follow from
\cite{Rou20},\cite{JF23} and \cite{Rou24-1}, where $K(O, n-1)$ was
only shown to be non-trivial when $O$ has a singular point.

\begin{cor}\label{free}
For $O\in {\cal C}$ and $n\geq 2$, $K(O, n-1)$ is free.
\end{cor}

Now we show how the proof of Theorem \ref{skr-jf} follows from
\cite{Rou20},\cite{JF23} and \cite{Rou24-1}.

\begin{proof}[Proof of Theorem \ref{skr-jf}]
  We divide the proof in three cases.

  \medskip
  \noindent
  {\bf Case 1.} Assume $O$ has no singular point, that is, it is a
  smooth $2$-manifold which is either the $2$-sphere with at least one
  puncture or a surface of genus $\geq 1$. Then, $\cal{O}_1(O)(=O)$ and 
  $O-\{\text{finitely many points}\}$ are both aspherical, that is
  all the higher homotopy groups are trivial. Next, by the
  Fadell-Neuwirth fibration theorem (Theorem \ref{FN}) $p:{\cal
    O}_n(O)\to {\cal O}_{n-1}(O)$ is a fibration with
  fiber $O-\{(n-1)\ \text{points}\}$. Therefore, by applying the
  long exact homotopy sequence induced by $p$ and an induction
  argument on $n$, we
  deduce that all the ${\cal O}_n(O)$ are aspherical. Hence we get an 
  exact sequence as in the statement of the theorem, with $K(O,
  n-1)=\l 1\r$.

  \medskip
  \noindent
  {\bf Case 2.} Let the underlying space of $O$ be a $2$-sphere with
  at least one puncture and with at least one cone 
  point. Therefore, $O$ is the complex plane with some punctures and
  at least one cone point. In this case, in [\cite{Rou20}, Theorem 2.14] we proved
  the theorem, which was then corrected in [\cite{JF23},
  Theorem C] by showing that in fact
  $K(O, n-1)\neq \l 1\r$ whenever $O$ has at least one cone point.

  \medskip
  \noindent
  {\bf Case 3.} If $O$ is orientable and has genus $\geq 1$, the
  theorem was proved in [\cite{Rou24-1}, Theorem 2.2]. Using Case 2
  above 
  and [\cite{Rou24-1}, Lemma 3.1],
  we also proved there that whenever $O$ has at least one cone point,
  $K(O, n-1)\neq \l 1\r$.

  The above arguments also show that if $K(O, n-1)= \l 1\r$, then
  $O$ can not have any singular point. 

  This completes the proof of the theorem.
\end{proof}
For the proof of Theorem \ref{mt}, together with Theorem \ref{skr-jf},
we also need the following general version
of the Snake Lemma.

\begin{lemma}\label{nasl} {\rm (Snake Lemma)}
  Consider the following commutative diagram of discrete groups with exact rows and columns.

\centerline{
    \xymatrix{
      &G_1\ar[r]\ar[d]^v&G_2\ar[d]^w\ar[r]^s&
    G_3\ar[r]\ar[d]^r&1\\
    1\ar[r]&H_1\ar[r]\ar[d]^u&H_2\ar[r]\ar[d]
    &H_3\ar[d]&\\
    &K_1\ar[r]^t\ar[d]&K_2\ar[r]\ar[d]&K_3\ar[d]&\\
    &1&1&1&}}

\centerline {{\rm Diagram 1}}

Then, there is a connecting homomorphism $\Delta:{\text {ker}}(r)\to K_1$ which makes the following
sequence exact.

\centerline{
  \xymatrix{{\text {ker}}(v)\ar[r]&{\text {ker}}(w)\ar[r]^{s|_{{\text
          {ker}}(w)}}&{\text {ker}}(r)\ar[r]^{\hspace{2mm}\Delta}&K_1\ar[r]^t&K_2\ar[r]&K_3.}}

\noindent
Consequently, if $t$ is injective, then $s|_{{\text {ker}}(w)}:{\text
  {ker}}(w)\to {\text {ker}}(r)$ is surjective. Also, if
$G_1\to G_2$ is injective, then so is ${\text {ker}}(v)\to {\text {ker}}(w)$.\end{lemma}

\begin{proof} The proof is a simple diagram chase. The lemma is well-known when the groups are 
  abelian. See [\cite{Jac91}, Ex. 1, p. 337]. The proof in the general
  case follows the same line of
  the proof of the abelian case. 
  The only difference to note here is that in the case of abelian
  groups the cokernels $K_i$, for $i=1,2,3$, 
  always exist, but in the non-abelian case we need this as an assumption.

  We give the definition of the connecting homomorphism and then leave the rest of the
  checking to the reader.

  Let $\alpha\in {\text {ker}}(r)$. Since $s$ is surjective
  there is $\beta\in G_2$ such that $s(\beta)=\alpha$. Clearly, using
  exactness at $H_1$, $w(\beta)\in H_1$.
  Define $\Delta(\alpha)=u(w(\beta))$. For well-definedness of $\Delta$, let
  $\beta'\in G_2$, such that $s(\beta')=\alpha$. Using exactness at $G_2$ and $H_1$, we see that
  $w(\beta^{-1}\beta')$ is the image of an element of $G_1$, and hence
  lies in the kernel of $u$. Therefore, $u(w(\beta))=u(w(\beta'))$.\end{proof}

\section{Proofs}
\begin{proof}[Proof of Theorem \ref{mt}] For $x\in S$, by $\bar x\in S/G$ we
  denote its orbit.
If $*_k=(x_1,x_2,\cdots, x_k)$, then $$F=p^{-1}(*_{n-1})=\{(x_1,x_2,\cdots,
x_{n-1},x)\in {\cal O}_n(S,G)\}.$$ 
  Therefore, since $F$ is homeomorphic to $S$ minus the union of the
  orbits $\bar x_i$, $i=1,2,\cdots, n-1$, 
  the action of $G$ induces an action on $F$.
  Let $\bar F=F/G$ and $\bar S=S/G$. 
  Note that, for all $k$, there is an induced action of $G^k$ on
  ${\cal O}_k(S,G)$. Then, the configuration space
  ${\cal O}_k(\bar S)$ (:=${\cal O}_k(\bar S, \l 1\r))$ is equal
  to ${\cal O}_k(S,G)/G^k$ as a subset of $\bar S^k$. Also if $q:G^n\to G^{n-1}$ is the
  projection to the first $n-1$ coordinates, then clearly $p$ is $q$-equivariant.
  Hence $p$ induces a map $\bar p:{\cal O}_n(\bar S)\to {\cal O}_{n-1}(\bar S)$.
  Similar statements are true for the inclusion map $i:F\to {\cal
    O}_n(S, G)$, since it is $s$-equivariant, where $s:G\to G^n$ is
  defined by $s(g)=(e,e,\cdots,e,g)$, $e\in G$ is the identity element. 
   One easily sees now that $i$ induces an inclusion map $\bar i:\bar F\to {\cal O}_n(\bar
  S)$, where $\bar F=\{(\bar x_1,\bar x_2,\cdots,
\bar x_{n-1},\bar x)\in {\cal O}_n(\bar S)\}$. Since for all $k$, $G^k$ also
   acts effectively and properly discontinuously on
   ${\cal O}_k(S,G)$, ${\cal O}_k(\bar S)$ is an orbifold. Hence we have orbifold
   covering maps $q:F\to \bar F$ and $q_k:{\cal O}_k(S, G)\to {\cal O}_k(\bar S)$, for
   all $k$.
   
 \medskip
  \centerline{
  \xymatrix{F\ar[r]^{\!\!\!\!\!\!\!\!\!\!\!\!\!\! i}\ar[d]^q&
    {\cal O}_n(S, G)\ar[d]^{q_n}\ar[r]^{\!\!\!\!\!\! p}&
    {\cal O}_{n-1}(S, G)\ar[d]\\
    \bar F\ar[r]^{\!\!\!\!\!\!\!\!\!\!\!\!\!\! \bar i}&
    {\cal O}_n(\bar S)\ar[r]^{\!\!\!\!\!\! {\bar p}}&
    {\cal O}_{n-1}(\bar S)}}

\centerline{Diagram 2}

   Then, $\bar i$ and $\bar p$ are maps of orbifolds (see Example
   \ref{map-example}).
   Hence, we have Diagram 2 
   in the category of orbifolds with the squares in the diagram commutative.

   Now, we apply the orbifold fundamental group $\pi_1^{orb}(-)$
functor on Diagram 2 and get Diagram 3,
with commutative squares.

Since the spaces on the top sequence are
manifolds,
these groups reduce to the ordinary fundamental groups. Here
$L(S, n-1)$ is the kernel of $i_*$.

   \medskip
  \centerline{
    \xymatrix@C-=0.35cm{1\ar[r]&L(S,n-1)\ar[r]\ar[d]
      &\pi_1(F, *_n)\ar@{..>}[r]^{\!\!\!\!\!\!\!\!\!\!\!\!\!\! i_*}\ar[d]^{q_*}&
    \pi_1({\cal O}_n(S, G), *_n)\ar[d]^{q_{n*}}\ar@{..>}[r]^{\!\!\!\!\!\! p_*}&
    \pi_1({\cal O}_{n-1}(S, G),*_{n-1})\ar@{..>}[r]\ar[d]&1\\
    1\ar[r]&K(\bar S, n-1)\ar[r]&\pi_1^{orb}(\bar F, \bar *_n)\ar[r]^{\!\!\!\!\!\!\!\!\!\! \bar i_*}&
    \pi_1^{orb}({\cal O}_n(\bar S), \bar *_n)\ar[r]^{\!\!\!\!\!\! {\bar p_*}}&
    \pi_1^{orb}({\cal O}_{n-1}(\bar S), \bar *_{n-1})\ar[r]&1}}

\centerline{Diagram 3}
\medskip
    
Note that, $\bar S$ is a $2$-dimensional orbifold, with isolated singular
points and hence they are all cone points (\cite{Sco83}, \S2). Furthermore, the
conditions on $S$ imply that $\bar S\in {\cal C}$. 
Hence, the exactness of the bottom sequence in Diagram 3   
 follows from Theorem \ref{skr-jf}. We have to show that the dotted
 part of the top
sequence in Diagram 3 is exact.

Next, we extend Diagram 3 to Diagram 4, where all the squares are commutative.

\centerline{
  \xymatrix@C-=0.35cm{&1\ar[d]&1\ar[d]&1\ar[d]&1\ar[d]&\\
    1\ar[r]&L(S,n-1)\ar[r]\ar[d]&\pi_1(F, *_n)\ar@{..>}[r]^{\!\!\!\!\!\!\!\!\!\!\!\!\!\! i_*}\ar[d]^{q_*}&
    \pi_1({\cal O}_n(S, G), *_n)\ar[d]^{q_{n*}}\ar@{..>}[r]^{\!\!\!\!\!\! p_*}&
    \pi_1({\cal O}_{n-1}(S, G),*_{n-1})\ar[d]\ar@{..>}[r]&1\\
    1\ar[r]&K(\bar S, n-1)\ar[r]\ar[d]&\pi_1^{orb}(\bar F, \bar *_n)\ar[r]^{\!\!\!\!\!\!\!\!\!\! \bar i_*}\ar[d]^r&
    \pi_1^{orb}({\cal O}_n(\bar S), \bar *_n)\ar[r]^{\!\!\!\!\!\! {\bar p_*}}\ar[d]^{r_n}
    &\pi_1^{orb}({\cal O}_{n-1}(\bar S), \bar *_{n-1})\ar[r]\ar[d]&1\\
    1\ar[r]&1\ar[r]&G\ar[r]^s\ar[d]&G^n\ar[r]^q\ar[d]&G^{n-1}\ar[r]\ar[d]&1\\
    &&1&1&1&}}

\centerline{Diagram 4}

The last three vertical sequences are exact by orbifold
covering space theory (see Example \ref{cov-map-example}).
See [\cite{Che01}, (4) in Proposition 2.2.3] for a more general statement.

  We now prove that the map 
  $L(S, n-1)\to K(\bar S, n-1)$ is an isomorphism.

  Note that this
homomorphism is the restriction of $q_*$, and hence is injective,
since $q_*$ is injective. To show that the image lies
in $K(\bar S, n-1)$, let $\alpha\in L(S, n-1)$, then since $L(S, n-1)$
is the kernel of $i_*$, $i_*(\alpha)=0$. Hence $q_{n*}(
i_*(\alpha))=\bar i_*(q_*(\alpha))=0$. Therefore, $q_*(\alpha)\in
K(\bar S,n-1)$.

Next, we show that $q_*:L(S, n-1)\to K(\bar S, n-1)$ is surjective.
Let $\bar\alpha\in K(\bar S, n-1)$. Then,
$\bar i_*(\bar\alpha)=0$ and hence $r_n(\bar
i_*(\bar\alpha))=s(r(\bar\alpha))=0$. Since $s$ is injective,
$r(\bar\alpha)=0$. Hence $\bar\alpha=q_*(\alpha)$ for some
$\alpha\in\pi_1(F, *_n)$. Consequently, $q_{n*}(i_*(\alpha))=\bar
i_*(q_*(\alpha))=0$. Since $q_{n*}$ is injective, this shows that
$i_*(\alpha)=0$. Hence $\alpha\in L(S,n-1)$. This shows that the
restriction $q_*:L(S,n-1)\to K(\bar S, n-1)$ is surjective.

Now to show that the top horizontal sequence in Diagram 3
is exact, we consider the following part of Diagram 4.

\centerline{
  \xymatrix{
    &\pi_1^{orb}(\bar F, \bar *_n)\ar[r]^{\!\!\!\! \bar i_*}\ar[d]&
    \pi_1^{orb}({\cal O}_n(\bar S), \bar *_n)\ar[r]^{\!\!\!\!\!\! {\bar p_*}}\ar[d]
    &\pi_1^{orb}({\cal O}_{n-1}(\bar S), \bar *_{n-1})\ar[r]\ar[d]&1\\
    1\ar[r]&G\ar[r]^s\ar[d]&G^n\ar[r]^q\ar[d]&G^{n-1}\ar[d]&\\
    &1\ar[r]\ar[d]&1\ar[r]\ar[d]&1\ar[d]&\\
  &1&1&1&}}

\centerline{Diagram 5}

Now, the exactness of the top sequence of
Diagram 4 follows immediately from Lemma \ref{nasl},
  by restricting it to the case of $K_1=K_2=K_3=\l 1\r$ and Diagram 5.

   When the action is also free, by Corollary \ref{corollary} 
  $i_*$ is injective, which implies $L(S, n-1)=\l 1\r$.

Conversely, if $L(S, n-1)=\l 1\r$, then by
Theorem \ref{skr-jf}, $\bar S$ has no singular point, and hence the
action of $G$ on $S$ is free.

This completes the proof of Theorem \ref{mt}.
\end{proof}

 \begin{proof}[Proof of Corollary \ref{identify}]
The proof of $L(S, n-1)\simeq K(\bar S, n-1)$ is deduced in the proof
of  Theorem \ref{mt}.\end{proof}

Next, we mention an extension of Theorem \ref{mt}.
Let $p_l:{\cal O}_n(S, G)\to {\cal O}_l(S, G)$ be the projection 
to the first $l$ coordinates, for some $1\leq l\leq n-1$ and $p_l(*_n)=*_l=(x_1,x_2,\ldots , x_l)$. Then,
$p_l^{-1}(*_l)$ is homeomorphic to ${\cal O}_{n-l}(S_l, G)$,
where $S_l=S-\cup_{j=1}^lGx_j$.
    
    \begin{thm}\label{gen}Theorem \ref{mt} is also
      valid for $p_l$ with $F$ replaced by $p_l^{-1}(*_l)$.\end{thm}

\begin{proof} Note that, Diagram 6 is also valid since 
  the middle horizontal sequence in the diagram is exact (\cite{Rou24-1}, Remark 3.2).

Here $\bar{S_l}=S_l/G$, $q_l:G^n\to G^l$ is
the projection to the first $l$-coordinates, $*^{n-l}=(x_{l+1}, x_{l+2},\ldots, x_n)$ and
$j:{\cal O}_{n-l}(S_l, G)\to {\cal O}_n(S, G)$
is defined by $$j(y_1,y_2,\ldots , y_{n-l})=(x_1,x_2,\ldots , x_l, y_1, y_2,\ldots , y_{n-l}).$$

\centerline{
  \xymatrix@C-=0.35cm{&1\ar[d]&1\ar[d]&1\ar[d]&1\ar[d]&\\
    1\ar[r]&L(S, l)\ar[r]\ar[d]&\pi_1({\cal O}_{n-l}(S_l, G), *^{n-l})\ar@{..>}[r]^{\hspace{3mm}j_*}\ar[d]&
    \pi_1({\cal O}_n(S, G), *_n)\ar[d]\ar@{..>}[r]^{{p_l}_*}&
    \pi_1({\cal O}_l(S, G),*_l)\ar[d]\ar@{..>}[r]&1\\
    1\ar[r]&K(\bar S, l)\ar[r]\ar[d]&\pi_1^{orb}({\cal O}_{n-l}(\bar{S_l}), \bar *^{n-l})\ar[r]^{\hspace{3mm}\bar j_*}\ar[d]&
    \pi_1^{orb}({\cal O}_n(\bar S), \bar *_n)\ar[r]^{{\bar {p_l}_*}}\ar[d]
    &\pi_1^{orb}({\cal O}_l(\bar S), \bar *_l)\ar[r]\ar[d]&1\\
    1\ar[r]&1\ar[r]&G^{n-l}\ar[r]\ar[d]&G^n\ar[r]^{q_l}\ar[d]&G^l\ar[r]\ar[d]&1\\
    &&1&1&1&}}

\centerline{Diagram 6}

Hence, we can apply Lemma \ref{nasl} again to conclude
that the dotted part of the top sequence 
in Diagram 6 is exact. 
The last part of Theorem \ref{mt} for $p_l$, also follows from a
similar argument.

\end{proof}

 Recall that if the action of $G$ on $X$ is properly
discontinuous and free, then by Theorem \ref{Xic} 
$p:{\cal O}_n(X,G)\to {\cal O}_{n-1}(X,G)$
is a locally trivial fibration. 
In the following proposition we prove that, in the
non-free action case, $p$ is not even a quasifibration. This proposition is
essentially [\cite{Rou20}, Proposition 2.11] with a correction in its hypothesis.

\begin{prop}\label{finallemma} Let $G$ be a finite group, 
  acting effectively on a connected 
  manifold $X$ of dimension $k \geq 2$, with isolated fixed points and
  has at least
    one fixed point. Assume that $X$ is either compact or has finitely generated
    $k$-th integral homology group. Then the map $p:{\cal O}_n(X,G)\to {\cal O}_{n-1}(X,G)$ is not
    a quasifibration.\end{prop}

  \begin{proof} Consider two points in ${\cal O}_{n-1}(X,G)$, one with
    trivial and another with non-trivial isotropy
    group. Then the fibers of $p$ on these two
    points are $X$ with 
    different numbers of points removed. Hence they have 
    non-isomorphic $k$-th homology groups, and therefore they
    are not weak homotopy equivalent. On the other hand,
    for a quasifibration over a path connected space 
    any two fibers are weak homotopy equivalent (\cite{Hat03},
    chap. 4, p. 479).\end{proof}

  We end with the following final remark.

  \begin{rem}{\rm Note that a quasifibration induces a long exact sequence
      of homotopy groups (\cite{Hat03}) as in the case of a fibration. Therefore, if
      ${\cal O}_{n-1}(S,G)$ is aspherical (or simply $\pi_2({\cal O}_{n-1}(S,G))=\l 1\r$)
      and the action has at
      least one fixed point, then since by Theorem \ref{mt} $L(S, n-1)\neq \l 1\r$,
      $i_*$ is not injective and hence $p$ is not a quasifibration. See Remark \ref{asphericityC}
    for some examples when this hypothesis is satisfied.}\end{rem}

\section{Appendix: The category of orbifolds}
In this section we give a short introduction to the category of
orbifolds and a couple of examples used in this note.
We refer to the articles \cite{Thu91}, \cite{Che01} and \cite{KL14} 
for some fundamentals on orbifolds.

\begin{defn}\label{orbi-defn}{\rm A second countable and Hausdorff
    topological space $M$ 
    is called an {\it orbifold} of dimension
    $n$, if for any $x\in M$, there is a connected open neighbourhood
    $U_x$ of $x$ and a connected open subset $V_x\subset {\Bbb R}^n$ with a
    finite group $G_x$ acting effectively on $V_x$, so that there is a
    homeomorphism $\sigma_x:V_x/G_x\to U_x$. $(U_x,V_x,G_x)$ is
    called a {\it coordinate chart} at $x$. $M$
    is also called the {\it underlying topological space} of the
    orbifold $M$. Here we use the same notation both for the orbifold
    as well as the underlying space, which will be clear from the
    context.}\end{defn}

\begin{defn}\label{map}{\rm (\cite{KL14}, \S2.1) Given two orbifolds $M_1$ and $M_2$, a
    {\it map of orbifolds}
    $f:M_1\to M_2$ is a continuous map of the underlying spaces, satisfying the following
    properties. Given any $x\in M_1$, there are coordinate charts
    $(U_x,V_x,G_x)$
    of $x$ and  $(U_{f(x)},V_{f(x)},G_{f(x)})$ of $f(x)$ so that
    $f(U_x)\subset U_{f(x)}$. Furthermore, there is a
    homomorphism $\rho_x:G_x\to G_{f(x)}$ and a $\rho_x$-equivariant
    smooth map $g_x:V_x\to V_{f(x)}$ which is a lifting of
    $f|_{U_x}$. That is, the following diagram is commutative. The
    vertical maps in the diagram are the quotient maps.
    
    \centerline{
      \xymatrix{V_x\ar[r]^{g_x}\ar[d]&V_{f(x)}\ar[d]\\
        U_x\ar[r]^{f|_{U_x}}&U_{f(x)}
        }}}\end{defn}

  \begin{defn}{\rm (\cite{Thu91}, Definition 5.3.1)   Let $f:M_1\to
      M_2$ be as in Definition \ref{map}. Then $f$ is
      called an {\it orbifold covering map} if it is
    surjective and for each $y\in M_2$, there is a chart $(U_y,V_y,G_y)$
    of $y$ and pairwise disjoint charts $(U_x,V_x,G_x)$ for $x\in
    f^{-1}(y)$, each satisfying the properties in Definition \ref{map}.
    In addition, we
    require that $f^{-1}(U_y)$ is equal to the union of $\{U_x\ |\ x\in
    f^{-1}(y)\}$ and for each $x\in f^{-1}(y)$, $\rho_x$ is
    injective and $U_x$ is homeomorphic to
    $V_y/\rho_x(G_x)$ making the following diagram commutative.

    \centerline{
      \xymatrix{U_x\ar[r]^{f|_{U_x}}&U_{y}\\
        V_y/\rho_x(G_x)\ar[r]\ar[u]&V_y/G_y\ar[u]^{\sigma_y}
        }}
  }\end{defn}

\begin{exm}\label{cov-map-example}{\rm Let $G$ be a group acting
    effectively and properly discontinuously on a connected manifold $M$. Then
    $M/G$ is an orbifold (\cite{Thu91}, Proposition 5.2.6) and the quotient map $M\to M/G$ is an
    orbifold covering map (\cite{Thu91}, p. 236). Furthermore, there is the following short
    exact sequence (\cite{Che01}, (4) in Proposition 2.2.3] or
    \cite{KL14}, \S2.12, p.110).

    $$1\to \pi_1(M)\to \pi_1^{orb}(M/G)\to G\to 1.$$}\end{exm}

\begin{exm}\label{map-example}{\rm (\cite{PR20}, Proposition 4.1,
    Remark 4.2)
    For $i=1,2$, let $G_i$ be a group 
    acting effectively and properly discontinuously on a smooth
    manifold $M_i$. Let
    $\rho:G_1\to G_2$ be a homomorphism and $f:M_1\to M_2$ is a
    $\rho$-equivariant smooth map. Then the induced map $M_1/G_1\to
    M_2/G_2$ is a map of orbifolds.}\end{exm}

\newpage
\bibliographystyle{plain}
\ifx\undefined\bysame
\newcommand{\bysame}{\leavevmode\hbox to3em{\hrulefill},}
\fi

\end{document}